\documentclass{birkjour}
 \newtheorem{thm}{Theorem}[section]
 \newtheorem{cor}[thm]{Corollary}
 \newtheorem{lem}[thm]{Lemma}
 \newtheorem{prop}[thm]{Proposition}
 \newtheorem{con}[thm]{Conjecture}
 \theoremstyle{definition}
 
 \theoremstyle{remark}

 \numberwithin{equation}{section}

\usepackage{amssymb}
\usepackage{enumerate}

\usepackage{cite}
\usepackage{url}
\usepackage{hyperref}

\DeclareMathOperator{\Capp}{Cap}
\DeclareMathOperator{\dist}{dist}
\DeclareMathOperator{\Mod}{Mod}

\begin{document}

%
%
%
%
%
%
%
%
%

\title[Sobolev abstract]{Vector-valued Sobolev spaces based on Banach function spaces}

\author[N. Evseev, \today]{Nikita Evseev}

\address{%
}

\email{evseev@math.nsc.ru}




\date{\today}

\begin{abstract}
It is known that for Banach valued functions there are several approaches to define a Sobolev class. We compare the usual definition via weak derivatives with the Reshetnyak-Sobolev space and with the Newtonian space; in particular, we provide sufficient conditions when all three agree. 
As well we revise the difference quotient criterion
and the property of Lipschitz mapping to preserve Sobolev space when it acting as a superposition operator.
\end{abstract}

\maketitle

\section{Introduction}
Our primary motivation behind this work is to provide non-differential characterization of Sobolev spaces.
In particular, this would supply us with tools for analysing functions valued in a family of Banach spaces, e.g. \cite{EM2020}. 
Such functions typically appear in the theory of evolution PDEs.
The other side of the work is that we consider Sobolev type spaces built upon a general Banach function norm.  

A general idea of our study is to make use of metrical analysis but taking into account the presence of a linear structure. 
The theory of Sobolev spaces on metric measure spaces is quite developed now.
For a detailed treatment and for references to the literature on
the subject, one may refer to the \cite{H2007} by J.~{Heinonen}, \cite{HP2000} by P.~{Haj{\l}asz} and P.~{Koskela}, 
and \cite{HKST01} by J.~Heinonen, P.~Koskela, N.~Shanmugalingam, J.T.~Tyson.

In the present paper, we study the Sobolev space of vector-valued functions $W^1X(\Omega;V)$ based on a Banach function space $X(\Omega)$, where $\Omega\subset\mathbb R^n$.
We discuss the connection with the Newtonian space $N^1X$ and with the Reshetnyak-Sobolev space $R^1X$ and 
provide sufficient conditions when $W^1X = R^1X = N^1X$.
More precisely, we prove that $W^1X=R^1X$ iff $V$ has the Radon-Nikod\'ym property,
whereas $R^1X=N^1X$ whenever the Meyers-Serrin theorem holds true for $W^1X(\Omega; \mathbb R)$.
Besides, we provide the difference quotient criterion and, as a consequence, obtain a version of pointwise description for Sobolev functions.  
Finally, we consider a question when Lipschitz mapping $f:V\to Z$ preserve a Sobolev class. It is always the case for $R^1X$ and $N^1X$,
while we should assume  that $Z$ enjoys the Radon-Nikod\'ym property to have inclusion $f(W^1X(\Omega; V))\subset W^1X(\Omega; Z)$.  

It happened that merely in the same time I. Caama\~{n}o, J.~A. Jaramillo, \'{A}. Prieto, and A.~Ruiz in \cite{CJPA2020} did the research on the subject.

\section{Preliminaries}
Let $\Omega\subset \mathbb R^n$ and
$M(\Omega)$ be the set of all real-valued measurable functions on $\Omega$.
A Banach space $X(\Omega)$ is said to be a Banach function space if it satisfies the following conditions:

\begin{enumerate}[(P1)]
\item if $|f|\leq g$ with $f\in M(\Omega)$ and $g\in X(\Omega)$, then $f\in X(\Omega)$ and $\|f\|_{X(\Omega)}\leq \|g\|_{X(\Omega)}$ (the lattice property);
\item if $0\leq f_n \nearrow f$ a.e., then $\|f_n\|_{X(\Omega)}\nearrow\|f\|_{X(\Omega)}$ (the Fatou property);
\item for any measurable set $A\subset\Omega$ with $|A|<\infty$ we have $\chi_A\in X(\Omega)$;
\item for any measurable set $A\subset\Omega$ with $|A|<\infty$ there exists a positive constant $C_A$ such that 
$\|f\|_{L^1(\Omega)} \leq C_A\|f\cdot\chi_{A}\|_{X(\Omega)}$ for all $f\in X(\Omega)$.
\end{enumerate}
When there is no ambiguity, we write $\|\cdot\|_X$ for $\|\cdot\|_{X(\Omega)}$.

Here we collect some notions and properties from the theory of Banach function spaces that are necessary.
For a comprehensive exposition of the theory, we refer the reader to book \cite{PKJF2013}.  

Let $\{A_n\}$ be a sequence of measurable subsets of $\Omega$,
we say $A_n\to\emptyset$ if $\chi_{A_n}\to 0$ a.e. on $\Omega$.
Banach function space $X(\Omega)$ has \textit{an absolutely continuous norm} if  
$\|f\cdot\chi_{A_n}\|_{X(\Omega)}\to 0$ whenever $A_n\to\emptyset$ for any $f\in X(\Omega)$.
(Examples $L^p$ ($1\leq p<\infty$), Lorentz $L^{p,q}$ ($1\leq q< \infty$),  see \cite[p. 216]{PKJF2013}.) 

Define the \textit{translation operator} $\tau_h$, with $h\in\mathbb R^n$ for $u\in M(\Omega)$ by 
\[
\tau_h u(x) =\begin{cases} u(x+h), & \text{ if } x+h\in\Omega,\\
			0, & \text{ if } x+h \not\in\Omega.
\end{cases}
\]
We say that $\|\cdot\|_{X(\Omega)}$ has \textit{the translation inequality property}
if for all $u\in X(\Omega)$ and all $h\in\mathbb R^n$ $\|\tau_hu\|_{X}\leq\|u\|_X$.
Note that every rearrangement invariant function norm over domain $\Omega$ possesses the translation inequality property.

Let $X(\Omega)$ be a Banach function space and let $X'(\Omega)$ be
its associate space. 
Then, for functions $u\in X(\Omega)$ and $v\in X'(\Omega)$
the following H\"older inequality holds
\[
\int_\Omega |uv|\, dx \leq \|u\|_X \|v\|_{X'}, 
\]
see \cite[Theorem 6.2.6]{PKJF2013}.
We will need the following Fatou lemma for Banach function spaces.
\begin{lem}[{\cite[Lemma 6.1.12]{PKJF2013}}]\label{lemma:Fatou}
Let $X(\Omega)$ be a Banach function space and assume that $f_n\in X(\Omega)$ and $f_n\to f$ a.e. on $\Omega$ for some $f\in M(\Omega)$.
Assume further that
\[
\liminf_{n\to\infty} \|f_n\|_X\leq\infty.
\]
Then $f\in X(\Omega)$ and
\[
\|f\|_X \leq \liminf_{n\to\infty} \|f_n\|_X.
\]
\end{lem}

\textit{Minkowski's integral inequality} for function norms $\|\|f(x,y)\|_Y\|_X \leq M\|\|f(x,y)\|_X\|_Y$
holds for all measurable functions and some fixed constant $M$ whenever there is  $p\in[1,\infty]$ such that $\|\cdot\|_X$ is $p$-concave
and $\|\cdot\|_Y$ is $p$-convex \cite{S1994}. In particular,  
\begin{equation}\label{eq:MInkowskii}
\bigg\|\int_Af(\cdot,y)\, dy\bigg\|_X \leq \int_A\|f(\cdot,y)\|_X\, dy.
\end{equation}

Also, we briefly provide some notions and facts from the analysis in Banach spaces.
Let $V$ be a Banach space.
A function $u:\Omega\to V$ is said to be strongly measurable if there is a sequence of simple functions
$u_k = \sum_{i=1}^{N_k}v_i\chi_{A_i}$, $v_i\in V$ such that $\|u - u_k\|_V\to 0$ a.e. on $\Omega$.
There is the theory of Bochner integral, which allows us to integrate vector-valued functions and supplies us with all necessary tools.  
By $X(\Omega; V)$, we denote the collection of all strongly measurable functions $u:\Omega\to V$ for which $\|u(\cdot)\|_V\in X(\Omega)$.
Together with the norm $\|u\|_{X(\Omega;V)} = \big\|\|u(\cdot)\|_V\big\|_{X(\Omega)}$, it becomes a Banach space (see \cite[p. 177]{L2004}). 
We say that $\tilde u$ is \textit{a representative} of $u$ if $u=\tilde u$ a.e.

There are several notions connected to absolute continuity that we use. 
A function $u:[a,b]\to V$ is said to be \textit{absolutely continuous}, if for any $\varepsilon>0$ there exists $\delta>0$ such that
$\sum_{i=1}^m\|u(b_i)-u(a_i)\|_{V} \leq \varepsilon$ 
for any collection of disjoint intervals $\{[a_i,b_i]\}\subset [a,b]$ such that
$\sum_{i=1}^m(b_i-a_i) \leq \delta$.
A function $u:\Omega\to V$ is said to be \textit{absolutely continuous on a curve}
$\gamma$ in $\Omega$ if $\gamma : [0, l(\gamma)]\to \Omega$ is rectifiable, parametrized by the arc
length, and the function $u\circ\gamma : [0, l(\gamma)] \to V$ is absolutely continuous.
A function $u:\Omega\to V$ is said to be \textit{absolutely continuous on lines} in $\Omega$ (belongs to $ACL(\Omega)$)
if $u$ is absolutely continuous on almost every compact line segment in $\Omega$ parallel to the coordinate axes.

Let $(\Omega, \Sigma, \mu)$ be a $\sigma$-finite complete measure space.
A Banach space $V$ has the \textit{Radon-Nikod\'ym property} (RNP) if for any measure 
$\nu:\Sigma\to V$ with bounded variation that is absolutely continuous with respect to $\mu$,
there exist a function $f\in L^1(\Omega; V)$ such that $\nu(A) = \int_{A}f\, d\mu$ for all $A\in\Sigma$. 
However, for our purposes we make use of equivalent descriptions for this property:
\begin{prop}[{\cite[Theorem 2.5.12]{HVVW2016}}]\label{prop:RNP}
For any Banach space $V$, the following assertions are equivalent:
\begin{enumerate}[(1)]
\item $V$ has the Radon-Nikod\'ym property;
\item every locally absolutely continuous function 
$f:\mathbb R\to V$ is differentiable almost everywhere;
\item every locally Lipschitz continuous function 
$f:\mathbb R\to V$ is differentiable almost everywhere.
\end{enumerate}
\end{prop}

Note that each reflexive space has the RNP, and so does every
separable dual. On the other hand there are spaces that do not have the RNP, such as $c_0$, $L^1([0,1])$. 
For more information on the RNP, see in \cite[Chapter 5]{BL2000}.

\section{Sobolev spaces based on Banach function spaces}
A function $v\in L^{1}_{loc}(\Omega; V)$ is said to be a weak partial derivative with respect to 
$j$th coordinate of the function  $u\in L^{1}_{loc}(\Omega; V)$ if
\[
\int_\Omega\frac{\partial\varphi}{\partial x_j}(x)u(x)\, dx = -\int_\Omega\varphi(x)v(x)\, dx
\]
for all $\varphi\in C^\infty_0(\Omega)$. In this case we denote $v=\partial_j u$.
The Sobolev space $W^1X(\Omega;V)$ is the space of all $u\in X(\Omega; V)$
whose weak derivatives exist and belong to $X(\Omega;V)$.
On $W^1X(\Omega;V)$ we define a norm 
\[
\|u\|_{W^1X} = \|u\|_{X(\Omega;V)} + \||\nabla u|\|_{X(\Omega)},
\]  
where 
$
|\nabla u| = \sqrt{\sum_{j=1}^{n}\|\partial_j u\|_V^2}.
$
In the case of real-valued functions we will use $W^1X(\Omega)$ instead of $W^1X(\Omega;\mathbb R)$.

If the norm $\|\cdot\|_X$ is absolutely continuous and has the translation inequality property,
then the Meyers-Serrin theorem holds true:
$C^{\infty}(\Omega; V)\cap W^1X(\Omega;V)$ is dense in $W^1X(\Omega;V)$ with respect to the norm $\|\cdot\|_{W^1X}$. 
In this case, Sobolev functions are approximated with the help of standard mollification technique \cite[Corollary 3.1.5]{F1995}. 

\textit{The Sobolev $X$-capacity} of a set $E\subset\Omega$ is defined as
\[
\Capp_{X}(E) = \inf\{\|u\|_{W^1X} \colon u \geq 1 \text{ on } E\}.
\]

\begin{thm}\label{theorem:qe}
Let $u_i \in C^{\infty}(\Omega)\cap W^1X(\Omega)$ and $\{u_i\}$ is a Cauchy sequence in $W^1X(\Omega)$.
Then there is a subsequence of $\{u_i\}$ that converges
pointwise in $\Omega$ except a set of $X$-capacity zero.
Moreover, the convergence is uniform outside a set of arbitrarily small $X$-capacity.
\end{thm}
\begin{proof}
This can be proved analogously to the $L^p$ case.
\end{proof}

\begin{thm}\label{theorem:W-ACL}
If $u\in W^1X(\Omega)$, then there is a representative $\tilde u$ which is absolutely continuous and differentiable almost everywhere on lines in $\Omega$.
Moreover, $\frac{\partial \tilde u}{\partial x_j} = \partial_j u$ a.e. 
\end{thm}
\begin{proof}
This follows from the fact that $W^1X(\Omega)\subset W^{1,1}_{loc}(\Omega)$.
\end{proof}

\subsection{Reshetnyak--Sobolev space}
In this subsection, we develop the ideas that we learned from \cite[Section 2]{HT2008} by Haj{\l}asz and Tyson.
We weaken their assumption of the separable conjugate of $V$ to the
Radon-Nikod\'ym property. 

The \textit{Reshetnyak--Sobolev space} $R^1X(\Omega;V)$ is the class of all functions $u\in X(\Omega;V)$ such that:
\begin{enumerate}[(A)]
\item for every $v^*\in V^*$, $\|v^*\|\leq 1$, we have $\langle v^*, u \rangle \in W^1X(\Omega)$;
\item there is a non-negative function $g\in X(\Omega)$ such that
\begin{equation}\label{eq:resh}
|\nabla\langle v^*, u \rangle| \leq g \quad \text{a.e. on } \Omega
\end{equation} 
for every $v^*\in V^*$ with $\|v^*\|\leq 1$.
\end{enumerate}   
A function g satisfying condition (B) above is called a \textit{Reshetnyak upper gradient} of $u$.
The norm in $R^1X(\Omega;V)$ is defined via
\[
\|f\|_{R^1X} = \|f\|_{X(\Omega;V)} + \inf\|g\|_{X(\Omega)},
\]
where the infimum is taken over all Reshetnyak upper gradients of $u$.

The form of the definition above is given by Yu. G. Reshetnyak (\cite[p. 573]{Reshetnyak97} for functions valued in a metric space); 
for functions valued in a Banach space, we refer to \cite{HKST01} and \cite{HT2008}.

In the next lemma, which is a modification of  \cite[Lemma 2.12]{HT2008}, we provide sufficient conditions for function $u$ to be in $W^1X(\Omega;V)$.
\begin{lem}\label{lem:lemma2.12}
Let $V$ be a Banach space enjoying the Radon-Nikod\'ym property.
Suppose function $u\in X(\Omega; V)$ is so that for every $j\in\{1,\dots, n\}$ 
 it has a representative $\tilde u$ which is absolutely continuous on almost every compact line segment in $\Omega$ parallel to $x_j$-axis
and partial derivatives exist and satisfy $\big\|\frac{\partial \tilde u}{\partial x_j}\big\|_V \leq g$ a.e. for some 
$g\in X(\Omega)$.
Then $u\in W^1X(\Omega;V)$ and 
$\|u\|_{W^1X}\leq \|u\|_{X(\Omega;V)} +  \sqrt{n}\|g\|_{X(\Omega)}$.
\end{lem}
\begin{proof}
Fix $j\in\{1,\dots, n\}$.
Due to the RNP, partial derivative $\frac{\partial \tilde u}{\partial x_j}$ exists 
on almost every compact line segment in $\Omega$ parallel to the coordinate axes (proposition \ref{prop:RNP}). 
Let $\Gamma$  be a collection of all segments in $\Omega$ parallel to the $x_j$-axis on which function $\tilde u$ fails to be absolutely continuous.
Denote $\Sigma = P_j\Gamma$ -- the projection of $\Gamma$ on subspace orthogonal to the $x_j$-axis, then $\mu^{n-1}(\Sigma)=0$. 
Now, with the help of the Fubini theorem, for any $\varphi\in C^\infty_0(\Omega)$ we have
\begin{multline*}
\int_{\Omega}u\frac{\partial\varphi}{\partial x_j}\, dx = \int_{\Omega}\tilde u\frac{\partial\varphi}{\partial x_j}\, dx 
= \int_{P_j\Omega} \int_{l_j(y)\cap\Omega}\tilde u\frac{\partial\varphi}{\partial x_j}\, ds\, dy\\
= \int_{P_j\Omega\setminus\Sigma} \int_{l_j(y)\cap\Omega}\tilde u\frac{\partial\varphi}{\partial x_j}\, ds\, dy
= \int_{P_j\Omega\setminus\Sigma} \int_{l_j(y)\cap\Omega}\frac{\partial\tilde u}{\partial x_j}\varphi\, ds\, dy
= \int_{\Omega}\frac{\partial\tilde u}{\partial x_j}\varphi\, dx,
\end{multline*}
where $l_j(y)$ is a line parallel to the $x_j$-axis and passing through $y\in P_j\Omega$. 
Therefore, $u$ have weak partial derivatives which are in $X(\Omega; V)$.
\end{proof}


\begin{lem}\label{lem:lemma2.13}
Let $u\in R^1X(\Omega;V)$. 
Then for each $j\in\{1,\dots, n\}$ there is a representative $\tilde u$ which 
is absolutely continuous on almost every compact line segment in $\Omega$ parallel to $x_j$-axis.
Moreover, the following limit exists and satisfies
\begin{equation}\label{eq:lemma2.13est}
\lim_{h\to 0}\frac{\|\tilde u(x+he_j)- \tilde u(x)\|_{V}}{h} \leq g(x) \quad \text{ for a.e. } x\in\Omega,
\end{equation} 
where $g\in X(\Omega)$ is a Reshetnyak upper gradient of $u$.
\end{lem}
\begin{proof}
The function $u\in R^1X(\Omega;V)$ is measurable; therefore, by the Pettis theorem, it is essentially separable valued.
In other words, there is a subset $Z\subset\Omega$ of measure zero so that $u(\Omega\setminus \Sigma_0)$ is separable in $V$.
Let $\{v_i\}_{i\in\mathbb N}$ be a dense subset in the difference set
\[
f(\Omega\setminus \Sigma_0) - f(\Omega\setminus \Sigma_0) = \{f(x) - f(y) \colon x,y\in\Omega\setminus \Sigma_0\},
\] 
and let $v_i^*\in V^*$, $\|v^*_i\|=1$, be such that $\|v_i\| = \langle v^*_i, v_i \rangle$ (the last is due to the the Hahn--Banach theorem, see \cite[p.17]{HKST2015}). 

For each $i\in\mathbb N$ there is a representative $u_i\in ACL(\Omega)$ of $\langle v^*_i, u \rangle \in W^1X(\Omega)$ (theorem \ref{theorem:W-ACL}),
and an inequality $|\nabla u_i|\leq g$ holds true.
Let $\Sigma_i\subset\Omega$ be a set of measure zero, where $u_i$ differs from $\langle v^*_i, u \rangle$.
Denote $\Sigma = \Sigma_0\cup\bigcup_i\Sigma_i$. 

Fix $j\in\{1,\dots,n\}$.
Then for almost all compact line segment $l:[a,b]\to\Omega$ of the form $l(\tau)=x_0+\tau e_j$ we have:
\begin{enumerate}[(a)]
\item  $g$ is integrable on $l$;
\item  $\mu^1(l\cap\Sigma) = 0$;
\item For each $i\in\mathbb N$ and every $a\leq s\leq t\leq b$
\begin{equation}\label{eq:lemma2.13-condition-c}
|u_i(x_0+te_j) - u_i(x_0+se_j)| \leq \int_s^tg(x_0+\tau e_j)\, d\tau.
\end{equation}
\end{enumerate}    
The Fubini theorem ensures (a) and (b), while (c) follows from the estimate $|\nabla u_i|\leq g$.
Let $l$ be a segment so that (a)-(c) hold true. 
If $x_0+se_j\not\in\Sigma$ and $x_0+te_j\not\in\Sigma$, then 
there is a sequence $v_{i_k}$ converging to $u(x_0+te_j) - u(x_0+se_j)$ in $V$.
It can be shown that in this case 
\[
\|u(x_0+te_j) - u(x_0+se_j)\|_V \leq \limsup_{k\to\infty}|u_{i_k}(x_0+te_j) - u_{i_k}(x_0+se_j)|.
\]
The last estimate together with  \eqref{eq:lemma2.13-condition-c} give us
\begin{equation}\label{eq:lemma2.13-l-Sigma}
\|u(x_0+te_j) - u(x_0+se_j)\|_V \leq \int_s^tg(x_0+\tau e_j)\, d\tau.
\end{equation}
If any of endpoints are in $\Sigma$, say $x_0+se_j\in\Sigma$, then we can choose a sequence $s_k\to s$ so that $x_0+s_ke_j\in l\setminus \Sigma$.
With the help of \eqref{eq:lemma2.13-l-Sigma}, it is easy to see that $u(x_0+s_ke_j)$ converges in $V$ and the limit does not depend on chose of sequence. 
This allows us to define the desired representative $\tilde u(x) = u(x)$ if $x\in\Omega\setminus Z$;
$\tilde u(x) = \lim\limits_{s_k\to 0} u(x+s_ke_j)$ if there is a segment with $x$ as its endpoint; 
and we put $\tilde u(x) = 0$ in other cases.
It easy to see that \eqref{eq:lemma2.13-l-Sigma} holds true for $\tilde u$, and almost every compact line segment in $\Omega$ parallel to $x_j$-axis. 	  
Estimate  \eqref{eq:lemma2.13est} follows immediately.
\end{proof}
We should note that in the lemma above the constructed representatives $\tilde u$ does not necessarily belong to $ACL(\Omega)$, 
but this does not affect our results.
However, it is possible to prove stronger property: there is a representative that is absolutely continuous on almost every rectifiable curve $\gamma$ in $\Omega$,
see \cite[Theorem 4.5]{CJPA2020} and the proof of \cite[Theorem 7.1.20]{HKST2015}. 

\begin{thm}\label{thm:W=R}
Let $\Omega\subset\mathbb R^n$ be open. 
\begin{enumerate}[1)]
\item If $u\in W^1X(\Omega;V)$, then $u\in R^1X(\Omega;V)$.
Moreover, $|\nabla u|$ is a Reshetnyak upper
gradient of $u$ and $\|u\|_{R^1X}\leq \|u\|_{W^1X}$.
\item If $V$ has the Radon-Nikod\'ym property and $u\in R^1X(\Omega;V)$, then $u\in W^1X(\Omega;V)$ and  $\|u\|_{W^1X}\leq \sqrt{n}\|u\|_{R^1X}$.
\end{enumerate} 
\end{thm}
\begin{proof}
1) Let $u\in W^1X(\Omega;V)$, and $v^*\in V^*$ with $\|v^*\|\leq 1$. 
Then $\langle v^*, u\rangle \in X(\Omega)$ since $|\langle v^*, u\rangle| \leq \|u\|_V$. 
Using the property of the Bochner integral that $\int \langle v^*, u\rangle = \langle v^*, \int u\rangle$,
it is easy to show that $\langle v^*, u\rangle$ has weak derivatives in $X(\Omega)$,
and $\partial_j \langle v^*, u\rangle = \langle v^*, \partial_j u\rangle$.
Moreover, $|\partial_j\langle v^*, u\rangle| \leq \|\partial_j u\|_V\leq |\nabla u|$ a.e. on $\Omega$.

2)  Let $u\in R^1X(\Omega;V)$. 
Then lemma \ref{lem:lemma2.13} and the RNP of $V$ imply the assumptions of lemma \ref{lem:lemma2.12}.
Thus, $u\in W^1X(\Omega;V)$.
\end{proof}

\begin{thm}\label{thm:RNP}
A Banach space $V$ has the Radon-Nikod\'ym property if and only if $R^1X(\Omega;V) = W^1X(\Omega;V)$.
\end{thm}
\begin{proof}
Due to theorem \ref{thm:W=R}, it remains to prove that $V$ has the RNP in the case 
$R^1X(\Omega;V)\subset W^1X(\Omega;V)$.
Let $f:I\to V$ be Lipschitz continuous, where $I$ is a bounded interval.
We may assume that $\Omega = I$ since we can embed $I^{d}$ into $\Omega$ and
treat the function $x\mapsto f(x_1)$.
For any $v^*\in V^*$ with $\|v^*\|\leq 1$ function $\langle v^*, f\rangle:I\to\mathbb R$ is
Lipschitz continuous with the same Lipschitz constant $L$ and its derivative 
$|\langle v^*, f\rangle'|\leq L$. 
As constant function $x\mapsto L$ belongs to $X(I)$, by lemma \ref{lem:lemma2.12} $\langle v^*, f\rangle \in W^1X(I)$.	
Thus, conditions (A) and (B) are fulfilled; therefore,  $f\in R^1X(I;V)$, 
by the assumption $f\in W^1X(I,V)\subset W^{1,1}_{loc}(I; V)$.
From the last fact, we obtain that the derivative $f'$ exist almost everywhere on $I$.
\end{proof}

There are other definitions of Reshetnyak--Sobolev space.
\begin{thm}
Let $u : \Omega\to V$ be a measurable function. Then the
following four conditions are equivalent:
\begin{enumerate}[(i)]
\item $u\in R^1X(\Omega;V)$.
\item There exists a non-negative function $\rho\in X(\Omega)$ with the
following property: for each $1$-Lipschitz function $\varphi:V\to\mathbb R$ function $\varphi\circ u \in W^1X(\Omega)$ and $|\nabla \varphi\circ u|\leq \rho$ a.e. on $\Omega$. 
\item There exists a non-negative function $\rho\in X(\Omega)$ with the
following property: for each $v\in u(\Omega)$ function $u_v(x) = \|u(x)- v\|_V$ belongs to $W^1X(\Omega)$ and $|\nabla u_v|\leq \rho$ a.e. on $\Omega$. 
\end{enumerate}
\end{thm}
\begin{proof}
(i)$\Rightarrow$(ii) follows from theorem \ref{thm:mapping2}. (ii)$\Rightarrow$(iii) is obvious.
To prove (iii)$\Rightarrow$(i), we use the same approach as in the proof o lemma \ref{lem:lemma2.13}.
For now we take a dense set $\{v_i\}_{i\in\mathbb N}$ in $u(\Omega\setminus\Sigma_0)$.
Let $\Sigma_i\subset\Omega$ be a set of measure zero, where $\|u-v\|_V$ differs from its absolutely continuous representative. 
Fix $j\in\{1,\dots,n\}$. Let $l(\tau) = x_0 +\tau e_j$ be a segment in $\Omega$ so that (a)-(c) hold true.
Then choosing a sequence $v_{i_k}\to u(x_0+te_j)$ we obtain
\begin{multline*}
|\langle v^*, u(x_0+t e_j) \rangle -\langle v^*, u(x_0+s e_j) \rangle| 
\leq \|u(x_0+t e_j) - u(x_0+s e_j)\|\\
=\lim_{k\to\infty}\big|\|u(x_0+t e_j) - v_{i_k}\|_V - \|u(x_0+s e_j) - v_{i_k}\|_V \big|
\leq \int_s^t\rho(x_0+\tau e_j)\, d\tau.
\end{multline*}
So there is a representative $u_{v^*}$ of $\langle v^*, u \rangle$, which is absolutely continuous on almost every compact line segment in $\Omega$ parallel to $x_j$-axis,
and its partial derivative exists and satisfies  $\big\|\frac{\partial u_{v^*}}{\partial x_j}\big\|_V \leq \rho$.
Due to lemma \ref{lem:lemma2.12} $\langle v^*, u \rangle\in W^1X(\Omega)$, and by the estimate above $|\nabla \langle v^*, u \rangle|\leq \sqrt{n}\rho$.
Thus, conditions (A) and (B) are realized. 
\end{proof}

\subsection{Newtonian space}
The concept of Newtonian spaces is based on the Newton--Leibniz formula 
and employs the idea of estimating the difference of function values in two distinct points by the integral over a curve that connects those points.  
An extensive study of Newtonian spaces $N^{1,p}$ could be found in \cite{HKST2015}. 
Whereas, in \cite{Maly2016}, L. Mal\'y  constructed the theory of Newtonian spaces based on quasi-Banach function lattices. 
Here we make use of elements of that theory taking into account that $X(\Omega)$, in particular, is a quasi-Banach function lattice.
  
\textit{$X$-modulus} of the family of curves $\Gamma$ is defined by 
\[
\Mod_X(\Gamma) = \inf\|\rho\|_{X(\Omega)},
\]
where the infimum is taken over all non-negative Borel functions $\rho$ that satisfy $\int_\gamma \rho\, ds\geq 1$ for all $\gamma\in\Gamma$ (such functions are called admissible densities for~$\Gamma$).

\begin{lem}[Estimates for cylindrical curve families]\label{lemma:cyl-est}
Consider a cylinder $G=X\times J$
where $E$ is a Borel set in $R^{n-1}$ with $\mu^{n-1}(E)<\infty$, and $J\subset\mathbb R$ is an interval of length $h\in(0,\infty)$. 
Let $\Gamma(E)$ be the family of all curves $\gamma_{y}:J\to G$, $\gamma_y(t) = (y, t)$ 
for $y\in E\setminus\Sigma$, with $\mu^{n-1}(\Sigma)=0$.
Then
\begin{equation}\label{eq:cyl-est}
\mu^{n-1}(E) \leq \|\chi_G\|_{X'}\cdot\Mod_X(\Gamma(E)) 
\end{equation}
and
\begin{equation}\label{eq:cyl-est2}
\Mod_X(\Gamma(E)) \leq \|\chi_G\|_{X}\cdot h^{-1}.
\end{equation}
\end{lem}
\begin{proof}
Let $\rho$ be an admissible density for $\Gamma(E)$.
By the Fubini theorem and H\"older's inequality we have
\[
\mu^{n-1}(E) \leq \int_{E}\int_{\gamma_y}\rho\, ds\, dy = \int_G\rho\, dx \leq \|\rho\|_X\cdot\|\chi_G\|_{X'}, 
\]
which implies  \eqref{eq:cyl-est}.
To obtain \eqref{eq:cyl-est2}, we observe that $\frac{1}{h}\cdot\chi_G$ is an admissible density for $\Gamma(E)$.
\end{proof}

The next lemma is a modification of \cite[Lemma 2.4]{CJPA2020}
\begin{lem}
Let $H$ be a hyperplane in $\mathbb R^n$ and $P_H:\mathbb R^n\to H$ be the orthogonal projector.
Suppose we are given some family $\Gamma$ consisting of 
line segments orthogonal to $H$.
If $\Mod_X(\Gamma) = 0$, then $\mu^{n-1}(P_H\Gamma) = 0$. 
\end{lem}
\begin{proof}
Let $w\in R^{n}$ be a unit normal of $H$.
Each curve in $\Gamma$ is of the form $\gamma_y=y+wt$, for some $y\in H$, and defined on some interval $a\leq t\leq b$.

For $k\in\mathbb N$, we are picking a subfamily $\Gamma_k$ in the following way:
for each $y\in B^{d-1}_k$ take one if any $\gamma\in\Gamma$ so that $P_H\gamma = y$ and 
is defined on interval $[a,b]\subset [-k,k]$, where $B^{n-1}_k$ is a $(n-1)$-ball in $H$ with radius $k$.
Denote $E_{k}:= P_H\Gamma_{k}$ and take a Borel set $\tilde E_{k}\supset E_{k}$ with the property $\mu^{n-1}(\tilde E_{k}) = \mu^{n-1}(E_{k})$.  
Consider an additional family $\tilde \Gamma_k$ consisting of curves $\gamma_y(t)= y+wt$ for each $y\in \tilde E_{k}$ defined on interval $[-k,k]$.
Then, $\tilde \Gamma_k$ and $\tilde E_{k}$ form a cylinder $G_{k}$ with base $\tilde E_{k}$ and height $2k$.

Therefore, due to estimate \eqref{eq:cyl-est}
\[
\mu^{n-1}( E_{k}) = \mu^{n-1}(\tilde E_{k}) \leq \|\chi_{G_{k}}\|_{X'}\cdot\Mod_X(\tilde \Gamma_{k})
\leq \|\chi_{G_{k}}\|_{X'}\cdot\Mod_X(\Gamma_{k}).
\]
So
\[
\mu^{n-1}(P_H\Gamma) \leq \sum_{k}\mu^{n-1}(\Gamma_{k}) = 0.
\]


\end{proof}
\begin{lem}[Fuglede's lemma]\label{lemma:Fuglede}
Assume that $g_k\to g$ in $X(\Omega)$ as $k\to\infty$. 
Then, there is a subsequence (which we still denote by $\{g_k\}$) such that
\[
\int_\gamma g_k\, ds \to \int_\gamma g\, ds \quad \text{ as } k\to\infty
\]
for $\Mod_X$-a.e. curve $\gamma$, while all the integrals are well defined and real-valued.
\end{lem}

\begin{lem}[{\cite[Proposition 5.10.]{M2013}}]\label{lemma:cap-mod-zero}
Let $E\subset\Omega$ be an arbitrary set,
define 
$
\Gamma_E = \{\gamma\in\Gamma(\Omega) \colon \gamma^{-1}(E) \ne \emptyset\}
$
--- the collection of all curves in $\Omega$ that meet $E$.
If $\Capp_X(E) = 0$, then $\Mod_X(\Gamma_E) = 0$.
\end{lem}

\textit{The Newtonian space} $N^1X(\Omega; V)$ consists of all functions $u\in X(\Omega; V)$ for which there is a non-negative Borel function $\rho\in X(\Omega)$ such that
\[
\|u(\gamma(0)) - u (\gamma(l_\gamma))\|_V 
\leq \int_{\gamma} \rho\, ds
\]
for $\Mod_X$-a.e. curve $\gamma$ in $\Omega$. Each such function $\rho$ is called \textit{$X$-weak upper gradient} of $u$.
Define a semi-norm on $N^1X(\Omega; V)$ via
\[
\|f\|_{N^1X} = \|f\|_{X(\Omega;V)} + \inf\|\rho\|_{X(\Omega)},
\]
where the infimum is over all $X$-weak upper gradients of $u$. 
Furthermore, we assume that $N^1X(\Omega; V)$ consists of equivalence classes of functions,
where $u_1\sim u_2$ means $\|u_1 - u_2\|_{N^1X} = 0$.
We write $N^1X(\Omega)$ instead of $N^1X(\Omega;\mathbb R)$. 


\begin{thm}\label{theorem:NinW}
Let $\Omega$ be a domain in $\mathbb R^n$ and $X(\Omega)$ be a Banach function space.

1)If $u\in N^1X(\Omega)$, then $u\in W^1X(\Omega)$ and $|\nabla u| \leq \sqrt{n}\rho$ a.e. on $\Omega$,
where $\rho$ is any $X$-weak upper gradient of $u$.

2) Suppose norm $\|\cdot\|_{X}$ is absolutely continuous and has the translation inequality property.
If $u\in W^1X(\Omega)$,  then there is a representative  $\tilde u\in N^1X(\Omega)$, and 
as a $X$-weak upper gradient of $\tilde u$, one can choose a Borel representative of $|\nabla u|$.
\end{thm}
\begin{proof}
1) Let $u\in N^1X(\Omega)$ and $\rho\in X(\Omega)$ be a $X$-weak upper gradient of $u$.
Function $u$ is absolutely continuous on $\Mod_X$-a.e. curve $\gamma$ in $\Omega$.
Thanks to lemma \ref{lemma:cyl-est}, $u$ is absolutely continuous on almost all lines parallel to coordinate axes.
Moreover, $\left|\frac{\partial u}{\partial x_j}\right|\leq \rho$ a.e. on such lines.
Thus, applying lemma \ref{lem:lemma2.12} we infer that $u\in W^1X(\Omega)$.

2) Let $u\in W^1X(\Omega)$, then there is a sequence of smooth functions $\{u_k\}$ so that
$u_k\to u$ and $\nabla u_k \to \nabla u$ in $X(\Omega)$, as $k\to\infty$.
For any curve $\gamma$ we have
\[
|u_k(\gamma(0)) - u_k(\gamma(l_\gamma))| \leq \int_{\gamma} |\nabla u_k|\, ds.
\]
Choose a Borel representative of $|\nabla u|$, then, by Fuglede's lemma \ref{lemma:Fuglede}
\[
\int_\gamma |\nabla u_k|\, ds \to \int_\gamma |\nabla u|\, ds \quad \text{ as } k\to\infty
\]
holds for $\Mod_X$-almost every curve.
Furthermore, due to theorem \ref{theorem:qe}, we can assume that $u_k\to u$ 
pointwise, except a set $E$ of capacity zero. 
On the other hand, by lemma \ref{lemma:cap-mod-zero}, $X$-modulus of the family of curves that meet $E$ is zero.
Therefore, we can pass to the limit and obtain that 
\[
|u(\gamma(0)) - u(\gamma(l_\gamma))| \leq \int_{\gamma} |\nabla u|\, ds
\]
holds for $\Mod_X$-almost every curve.
\end{proof}

\begin{thm}
Let $\Omega$ be a domain in $\mathbb R^n$, $V$ be a Banach space, and $X(\Omega)$ be a Banach function space.

1) If $u\in N^1X(\Omega; V)$, then $u\in R^1X(\Omega; V)$ and 
$\sqrt{n}\rho$ is its Reshetnyak upper gradient, where
$\rho$ is arbitrary $X$-weak upper gradient of $u$.

2) Suppose norm $\|\cdot\|_{X}$ is absolutely continuous and has the translation inequality property.
If $u\in R^1X(\Omega;V)$,  then there is a representative  $\tilde u\in N^1X(\Omega;V)$ 
and as a $X$-weak upper gradient of $\tilde u$ one can choose a Borel representative of any  Reshetnyak upper gradient of $u$.
\end{thm}
\begin{proof}
1) Let $\rho$ be a $X$-weak upper gradient of $u$.
For any $v^*\in V^*$ with $\|v^*\|\leq 1$ and curve $\gamma$, we have
\[
|\langle v^*, u \rangle(\gamma(0)) - \langle v^*, u \rangle(\gamma(l_\gamma))| 
\leq \|u(\gamma(0)) - u (\gamma(l_\gamma))\| 
\leq \int_{\gamma} \rho\, ds
\]
Therefore $\langle v^*, u \rangle \in N^1X(\Omega)$ with $X$-weak upper gradient $\rho$ not depending on $v^*$.
Due to theorem \ref{theorem:NinW}, $\langle v^*, u \rangle \in W^1X(\Omega)$ and $|\nabla\langle v^*, u \rangle| \leq \sqrt{d}\cdot\rho$.
So $u\in R^1X(\Omega; V)$. 

2) Let $u\in R^1X(\Omega; V)$ and $g\in X(\Omega)$ be its Reshetnyak upper gradient.
Then due to theorem \ref{theorem:NinW}, for any $v^*\in V^*$ with $\|v^*\|\leq 1$, 
function $\langle v^*, u \rangle$ has a representative in $N^1X(\Omega)$.
Moreover, a Borel representative of $g$ is a $X$-weak upper gradient for each of those representatives above (not depending on $v^*$).
Therefore, to construct the desired representative of $u$, we can proceed as in the proof of lemma \ref{lem:lemma2.13}
(also see the proof from \cite[p.182-183]{HKST2015}).

\end{proof}

\subsection{Description via difference quotients}
Here we extend the characterization of Sobolev spaces via difference quotients known for $L^p$-spaces to the case of Banach function spaces. 
For the real-valued case see \cite[Theorem 2.1.13]{Bogachev2010} and \cite[Proposition 9.3]{Brezis2011},
and for vector case see  \cite[Proposition 2.5.7]{HVVW2016} and \cite[Theorem 2.2]{AK2018}.
\begin{thm}\label{theorem:DQC1}
Let $\Omega\subset\mathbb R^n$ and $X(\Omega)$ have the Radon-Nikod\'ym property.
If $u\in X(\Omega)$ and there is a constant $C\in[0,\infty)$ such that
\begin{equation}\label{eq:theorem:DQC1:1}
\|\tau_{te_j}u - \tau_{se_j}u\|_{X(\omega)} \leq C|t-s|, \quad j\in{1,\dots, n}
\end{equation}
for all $\omega \Subset \Omega$ with $\max\{|t|, |s|\} < \dist(\omega,\partial\Omega)$,
then $u\in W^1X(\Omega)$ and $\|\nabla u\|_{X(\Omega)}\leq nC$.
\end{thm}
\begin{proof}
Fix $j\in{1,\dots,n}$ and let $\omega\Subset\Omega$ be bounded.
First, we prove that weak derivatives of $u|_\omega$ exist in $X(\omega)$ and their norms are bounded by $C$.
Let $\omega\Subset\omega'\Subset\Omega$ and $0<\delta<\dist(\omega',\partial\Omega)$.
Consider function $G:(-\delta, \delta)\to X(\omega')$ defined by the rule $t\mapsto\tau_{te_j}u$.
By the assumption we have
\[
\|G(t)-G(s)\|_{X(\omega')} = \|\tau_{te_j}u - \tau_{se_j}u\|_{X(\omega)} \leq C|t-s|,
\]
meaning that $G$ is Lipschitz continuous. 
Due to the RNP of $X$, mapping $G$ is differentiable a.e.
Then fix $0\leq t_0<\dist(\omega, \partial\omega')$ so that 
\begin{equation}\label{eq:G'}
G'(t_0) = \lim_{h\to 0}\frac{u(\cdot + (t_0+h)e_j) - u(\cdot + t_0e_j)}{h}
\end{equation} 
exists in $X(\omega')$. Choose a sequence $h_k\to 0$ such that limit \eqref{eq:G'} exist a.e. in $\omega'$,
and in particular in $\omega-t_0e_j\subset\omega'$. 
For $x\in\omega$, we define 
\[
g_\omega(x) := \lim_{k\to\infty}\frac{u(x+h_ke_j) - u(x)}{h_k}.
\]
Then $g_\omega$ is measurable, and by lemma \ref{lemma:Fatou} 
$g_\omega\in X(\omega)$ with $\|g_\omega\|_{X(\omega)}\leq C$.
Denote $g^k_\omega(x) = \frac{u(x+h_ke_j) - u(x)}{h_k}$ and show that for any $\varphi\in C_0^\infty(\omega)$ the next equality holds 
\[
\lim_{k\to\infty}\int_\omega g^k_\omega(x)\varphi(x)\, dx = \int_\omega g_\omega(x)\varphi(x)\, dx
\]
Indeed:
\begin{multline*}
\bigg|\int_\omega g^k_\omega(x)\varphi(x)\, dx - \int_\omega g_\omega(x)\varphi(x)\, dx\bigg|
\leq \int_\omega |g^k_\omega(x)-g_\omega(x)|\cdot|\varphi(x)|\, dx\\
= \int_{\omega-t_0e_j} |g^k_\omega(y+t_0e_j)-G'(t_0)(y)|\cdot|\varphi(y+t_0e_j)|\, dy\\
\leq \bigg\|\frac{u(\cdot + (t_0+h_k)e_j) - u(\cdot + t_0e_j)}{h_k} - G'(t_0) \bigg\|_{X(\omega')}\|\varphi(\cdot +t_0e_j)\|_{X'(\omega')}
\to 0. 
\end{multline*}
We deduce that $g$ is a weak derivative:
\begin{multline*}
\int_\omega g_\omega(x)\varphi(x)\, dx
= \lim_{k\to\infty}\int_\omega \frac{u(x+h_ke_j) - u(x)}{h_k}\varphi(x)\, dx \\
= \lim_{k\to\infty}\int_\omega \frac{\varphi(x+h_ke_j) - \varphi(x)}{h_k}u(x)\, dx
= \int_\omega u(x)\frac{\partial\varphi}{\partial x_j}(x)\, dx.
\end{multline*}
Now we take a monotone sequence of bounded domains 
$\omega_n\Subset\omega_{n+1}\Subset\Omega$
such that $\bigcup_{n}\omega_n = \Omega$.
Functions $g_{\omega_n}$ agree on the intersections
of their supports; therefore, they can be pieced together to a globally
defined measurable function $g$. 
Again, thanks to lemma \ref{lemma:Fatou}, $g \in X(\Omega)$ and $\|g_\omega\|_{X(\Omega)}\leq C$.  
In the same manner as above, we derive that $g = \partial_j u$ on $\Omega$. 
\end{proof}

\begin{thm}\label{theorem:DQC2}
Let $\Omega\subset\mathbb R^n$, $V$ be a Banach space, and $X(\Omega)$ has the Radon-Nikod\'ym property.
If $u\in X(\Omega;V)$ and there is a constant $C\in[0,\infty)$ such that
\[
\|\tau_{te_j}u - \tau_{se_j}u\|_{X(\omega;V)} \leq C|t-s|, \quad j\in{1,\dots, n}
\]
for all $\omega \Subset \Omega$ with $\max\{|t|, |s|\} < \dist(\omega,\partial\Omega)$,
then $u\in R^1X(\Omega; V)$, and there is $g$ a Reshetnyak upper gradient of $u$ so that
$\|g\|_{X(\Omega)}\leq nC$.
\end{thm}
\begin{proof}
For any $v^*\in V^*$ with $\|v^*\|\leq 1$, it is clear that $\langle v^*, u \rangle \in X(\Omega)$.
Have the following estimate
\begin{multline*}
|\tau_{te_j}\langle v^*, u \rangle(x) - \tau_{se_j}\langle v^*, u \rangle(x)|
= |\langle v^*, \tau_{te_j}u(x) \rangle - \langle v^*, \tau_{se_j}u(x) \rangle|\\
= |\langle v^*, \tau_{te_j}u(x) -  \tau_{se_j}u(x) \rangle|
\leq \|\tau_{te_j}u(x) - \tau_{se_j}u(x)\|_{V}.
\end{multline*}
Then, for any $\omega \Subset \Omega$ with $\max\{|t|, |s|\} < \dist(\omega,\partial\Omega)$, we have
\[
\|\tau_{te_j}\langle v^*, u \rangle - \tau_{se_j}\langle v^*, u \rangle\|_{X(\omega)}
\leq \|\tau_{te_j}u - \tau_{se_j}u\|_{X(\omega;V)} 
\leq C|t-s|.
\]
Thus, all the assumptions of theorem \ref{theorem:DQC1}  are fulfilled. 
So $\langle v^*, u \rangle \in W^1X(\Omega)$.
Now, find a majorant. Define 
\[
g_j(x) := \liminf_{h\to 0}\frac{\|u(x+he_j) - u(x) \|_V}{|h|},
\] 
which belongs to $X(\Omega)$ and $\|g_j\|_{X(\Omega)}\leq C$ (due to lemma \ref{lemma:Fatou}). 
Applying the next estimate 
\[
\frac{|\langle v^*, u(x+h) \rangle - \langle v^*, u(x) \rangle|}{|h|}
\leq \frac{\|u(x+he_j) - u(x) \|_V}{|h|},
\]
we derive that
\begin{multline*}
|\partial_j\langle v^*, u \rangle(x)| 
=\lim_{h\to 0}\frac{|\langle v^*, u(x+h) \rangle - \langle v^*, u(x) \rangle|}{|h|}\\
\leq  \liminf_{h\to 0}\frac{\|u(x+he_j) - u(x) \|_V}{|h|}
= g_j(x).
\end{multline*}
So $g = \sqrt{\sum g_j^2}$ is a Reshetnyak upper gradient of $u$, and the estimate 
$\|g\|_{X(\Omega)} \leq \sum\|g_j\|_{X(\Omega)} \leq nC$ holds true.
\end{proof}

\begin{thm}
Let $\Omega\subset\mathbb R^n$, $V$ be a Banach space, and $X(\Omega)$ be a Banach function space.

1) Suppose norm $\|\cdot\|_{X}$ is absolutely continuous and has the translation inequality property.
If $u\in W^1X(\Omega;V)$, then
\begin{equation}\label{eq:theorem:DQC3:1}
\|\tau_{te_j}u - \tau_{se_j}u\|_{X(\omega;V)} \leq \|\partial_ju\|_{X(\Omega;V)}|t-s|, \quad j\in{1,\dots, n}
\end{equation}
for all $\omega \Subset \Omega$ with $\max\{|t|, |s|\} < \dist(\omega,\partial\Omega)$.
 
2) Suppose $X$ and $V$ have the Radon-Nikod\'ym property. 
If $u\in X(\Omega;V)$ and there is a constant $C\in[0,\infty)$ such that
\begin{equation}\label{eq:theorem:DQC3:2}
\|\tau_{te_j}u - \tau_{se_j}u\|_{X(\omega;V)} \leq C|t-s|, \quad j\in{1,\dots, n}
\end{equation}
for all $\omega \Subset \Omega$ with $\max\{|t|, |s|\} < \dist(\omega,\partial\Omega)$,
then $u\in W^1X(\Omega; V)$ and $\|\nabla u\|_{X(\Omega)}\leq nC$.
\end{thm}
\begin{proof}
1) By the density it is sufficient to consider $u\in C^{\infty}(\Omega; V)\cap W^1X(\Omega;V)$.
Then,
\[
u(x+te_j) - u(x+se_j) = \int_s^t \frac{d}{dr}u(x+re_j)\, dr = \int_s^t \frac{\partial}{\partial x_j}u(x+re_j)\, dr. 
\] 
Applying Minkowski's inequality \eqref{eq:MInkowskii} 
and then the translation inequality property, we derive \eqref{eq:theorem:DQC3:1}. 

2) It is a consequence of theorems \ref{thm:W=R} and  \ref{theorem:DQC2}
\end{proof}

In \cite{AK2018}, W. Arendt and M. Kreuter and obtained the following characterization of the Radon-Nikod\'ym property:
\textit{A Banach space $V$ has the RNP  iff the difference quotient criterion \eqref{eq:theorem:DQC3:2} characterizes the space $W^{1,p}(\Omega; V)$, $p\in(1,\infty]$}.
We are interested whether there exists such kind of property for a base space $X(\Omega)$. 
Namely, we suppose that the following would be reasonable.
\begin{con}
If the difference quotient criterion \eqref{eq:theorem:DQC1:1} characterizes the space $W^{1}X(\Omega)$, then
a Banach function space $X(\Omega)$ has the Radon-Nikod\'ym property.
\end{con}
At least for $L^p$-spaces, it is true.	

\subsection{A maximal function characterization}
Another fruitful observation consist in pointwise description of Sobolev functions via maximal function.

\begin{thm}\label{theorem:Maximal}
Let $\Omega\subset\mathbb R^n$, $V$ be a Banach space, $X(\Omega)$ have the Radon-Nikod\'ym property,
and norm $\|\cdot\|_{X(\Omega)}$ have the translation inequality property.
If $u\in X(\Omega;V)$ and there is a non-negative function $h\in X(\Omega)$ such that
\begin{equation}\label{eq:Maximal}
\|u(x) - u(y)\|_{V} \leq |x-y|(h(x)+h(y)), \quad \text{ a.e. on }  \Omega, 
\end{equation}
then $u\in R^1X(\Omega; V)$ and $\|g\|_{X(\Omega)}\leq 2n\|h\|_{X(\Omega)}$, 
where $g$ is some Reshetnyak upper gradient of $u$.
\end{thm}
\begin{proof}
For all $j=1,\dots,n$ and any $\omega \Subset \Omega$ with $\max\{|t|, |s|\} < \dist(\omega,\partial\Omega)$,
taking into account the translation inequality property, we deduce 
\begin{multline*}
\|\tau_{te_j}u - \tau_{se_j}u\|_{X(\omega;V)} \leq |t-s|\cdot\|\tau_{te_j}h + \tau_{se_j}h\|_{X(\omega)}\\ 
\leq  |t-s|\cdot 2\|h\|_{X(\Omega)}.
\end{multline*}
By theorem \ref{theorem:DQC2}, we conclude that $u\in R^1X(\Omega; V)$. 
\end{proof}

\begin{cor}\label{cor:Maximal}
In the assumptions of theorem \ref{theorem:Maximal} suppose that $V$ has the Radon-Nikod\'ym property.
Then it follows that $u\in W^1X(\Omega; V)$ and $\|\nabla u\|_{X(\Omega)}\leq 2n\|g\|_{X(\Omega)}$.
\end{cor}

A sufficiency counterpart to the theorem \ref{theorem:Maximal} (and to corollary \ref{cor:Maximal})
sounds in the following way:
\begin{thm}
Let $\Omega\subset\mathbb R^n$, $V$ be a Banach space, $X(\Omega)$ be a Banach function space
 such that Hardy-Littlewood maximal operator $M$ is bounded in $X(\Omega)$.
If $u\in W^1X(\Omega;V)$, then 
\[
\|u(x) - u(y)\|_{V} \leq C|x-y|\big(M(|\nabla u|)(x)+M(|\nabla u|)(y)\big)
\]
holds for some constant $C$ and almost all $x,y\in\Omega$ with $B(x, 3|x-y|) \subset\Omega$.
\end{thm}

This result has been recently obtained in \cite{JMSV2020} by P. Jain, A. Molchanova, M. Singh, and S. Vodopyanov for the real-valued case.    
It is easy to see that the proof of \cite[Theorem 2.2]{JMSV2020} works for vector-valued functions as well.

\section{Mapping theorems}

If $u\in N^1X(\Omega; V)$ and $f:V\to Z$ is Lipschitz continuous with $f(0)=0$, then it is obvious that $f\circ u \in N^1X(\Omega; Z)$
and $\operatorname{Lip}(f)\rho$ is its $X$-weak upper gradient. 
Here we discuss superpositions of Lipschitz mapping and functions from classes $W^1X$ and $R^1X$.

\begin{thm}\label{thm:mapping1}
Suppose that $V,Z$ are Banach spaces such that $Z$ has the Radon-Nikod\'ym property, and $X(\Omega)$ is a Banach function space.
Let $f:V\to Z$ be Lipschitz continuous and assume that $f(0)=0$ if $|\Omega|=\infty$.
Then $f\circ u\in W^1X(\Omega;Z)$ for any $u\in W^1X(\Omega;V)$. 
\end{thm}
\begin{proof}
Let $u\in W^1X(\Omega;V)$. 
There is a representative $\tilde u$ which is absolutely continuous on lines in $\Omega$;
then the same holds for $f\circ \tilde u$, which is a representative of $f\circ u$.
Due to the RNP of $Z$ there exist partial derivatives $\frac{\partial f\circ\tilde u}{\partial x_j}$.
For almost all $x\in\Omega$ we have
\begin{multline*}
\bigg\|\frac{\partial f \circ\tilde u}{\partial x_j}(x) \bigg\|_Z 
= \lim_{h\to 0}\frac{\|f\circ \tilde u(x+he_j) - f\circ \tilde u(x)\|_Z}{|h|}\\ 
\leq \lim_{h\to 0} L\frac{\| \tilde u(x+he_j) - \tilde u(x)\|_V}{|h|}
= L\bigg\|\frac{\partial \tilde u}{\partial x_j}(x) \bigg\|_V = L\|\partial_j u(x)\|_V,
\end{multline*}
where $L=\operatorname{Lip}(f)$.
Let $g(x) = L|\nabla u(x)|$, then, by lemma \ref{lem:lemma2.12}, $f\circ u$ belongs to $W^1X(\Omega;Z)$.
\end{proof}

\begin{thm}\label{thm:mapping2}
Let $\Omega\subset\mathbb R^{n}$ be open, $V$, $Z$ be Banach spaces,
and $f:V\to Z$ a Lipschitz continuous mapping ($f(0)=0$ in the case $|\Omega|=\infty$).
Then, $f\circ u \in R^1X(\Omega;Z)$ whenever $u \in R^1X(\Omega;V)$.
\end{thm}
\begin{proof}
Let $u \in R^1X(\Omega;V)$, and $z^*\in Z^*$ with $\|z^*\|\leq 1$. 
It is clear that $f\circ u\in X(\Omega;V)$.
Define function $\psi:V\to\mathbb R$ by the rule $\psi(v) = \langle z^*, f(v) \rangle$.
Then, $\psi$ is Lipschitz continuous, and $\langle z^*, f\circ u \rangle = \psi\circ u$.
With the help of theorem \ref{thm:mapping1}, the last guarantees  $\langle z^*, f\circ u \rangle \in W^1X(\Omega)$ 
and $|\nabla \langle z^*, f\circ u \rangle| \leq Lg$.
Thus, we conclude that $f\circ u \in R^1X(\Omega;Z)$.
\end{proof}

\begin{thm}\label{thm:mapping2RNP}
Let $\Omega\subset\mathbb R^{n}$ be open, and $V$, $Z$ be Banach spaces, $V\ne\{0\}$.
If for any Lipschitz mapping $f:V\to Z$ we have $f\circ u \in W^1X(\Omega;Z)$ whenever $u \in W^1X(\Omega;V)$,
then $Z$ has the Radon-Nikod\'ym property.
\end{thm}
\begin{proof}
Suppose $Z$ does not have the RNP.
Then there is a Lipschitz function $h:[a,b]\to Z$, which is not differentiable almost everywhere.
Fix elements $v_0\in V$ and $v^*_0\in V^*$ so that $\langle v_0^*, v_0 \rangle=1$.
Consider the next function $f(v)=h(\langle v_0^*, v \rangle)$; it is clear that $f:V\to Z$ is Lipschitz continuous.
We can assume that $Q=[a,b]^n\Subset\Omega$.
Choose function $\eta\in C_0^\infty(\Omega)$ such that $\eta(x)=1$ when $x\in Q$.
Then, we define function $u(x)=v_0\cdot x_1\eta(x)$, which is in $W^1X(\Omega;V)$.
Therefore, by the assumption $f\circ u\in W^1X(\Omega;V)$.  
On the other hand, $f\circ u(x) = h(x_1)$ when $x\in Q$, 
meaning that $f\circ u$ is not differentiable almost everywhere on intervals in $Q$,
and this contradicts to theorem \ref{theorem:W-ACL}.  
 
\end{proof}

The following lemma is similar to \cite[Corollary 3.4. and Corollary 3.4.]{AK2018}.
\begin{lem}\label{lem:norm}
If $u\in R^1X(\Omega;V)$, then $\|u(\cdot)\|_V \in W^1X(\Omega)$ and
$|\partial_j\|u(\cdot)\|_V| \leq g$ a.e.,
where $g\in X(\Omega)$ is a Reshetnyak upper gradient of $u$. 
\end{lem}
\begin{proof}
Let $u\in R^1X(\Omega;V)$, then, by theorem \ref{thm:mapping2},  
$\|u(\cdot)\|_V \in R^1X(\Omega)=W^1X(\Omega)$.
Then, with the help of lemma \ref{lem:lemma2.13}, we infer
$$
\lim_{h\to 0}\frac{|\|u(x+he_j)\|_V - \|u(x)\|_V |}{|h|}
\leq \lim_{h\to 0}\frac{\|u(x+he_j) - u(x)\|_V }{|h|}
\leq g(x)
$$
for almost all $x\in\Omega$.
\end{proof}

\begin{thm}\label{thm:emb}
Let $\Omega\subset\mathbb R^n$ be open such that we have 
a continuous embedding $W^1X(\Omega)\hookrightarrow Y(\Omega)$ for some Banach function space $Y(\Omega)$. 
Then we also have a continuous embedding $W^1X(\Omega;V)\hookrightarrow Y(\Omega;V)$.
\end{thm}
\begin{proof}
Let $u\in W^1X(\Omega;V)$. 
Then by lemma \ref{lem:norm} $\|u\|_V\in W^1X(\Omega)$, and by the assumption 
$\|u\|_V\in Y(\Omega)$. The last implies $u\in Y(\Omega; V)$.

Now let $C$ be the norm of the real-valued embedding.
Again, using lemma \ref{lem:norm}, we derive
$$
\|u\|_{Y(\Omega; V)} = \big\|\|u\|_V \big\|_{Y(\Omega)}
\leq C \big\|\|u\|_V \big\|_{W^1X(\Omega)}
\leq C \sqrt{n}\|u\|_{W^1X(\Omega;V)}.
$$
\end{proof}




\bibliographystyle{abbrv}
\bibliography{kothe}

\end{document}